\documentclass[11pt,letterpaper]{amsart}

\usepackage{fancyhdr}
\usepackage{bm}
\usepackage{hyperref}
\usepackage{graphicx} 
\usepackage{nicematrix}
\usepackage{amsthm,amssymb,amsmath}
\usepackage[T1]{fontenc}
\usepackage[utf8]{inputenc}
\usepackage{hyperref}
\usepackage{lipsum}
\usepackage{mathrsfs}
\usepackage{enumitem}
\usepackage{stmaryrd}
\usepackage{setspace}
\usepackage{yfonts}
\usepackage{float}
\usepackage{mathtools}
\usepackage{parskip}
\usepackage[all,cmtip]{xy}
\usepackage{tikz-cd}
\tikzcdset{row sep/normal=50pt, column sep/normal=50pt}

\newtheorem{lemma}{Lemma}[section]

\newtheorem{remark}[lemma]{Remark}
\newtheorem{theorem}[lemma]{Theorem}
\newtheorem{theorem*}{Theorem}

\newtheorem{example*}[lemma]{Example}
\newtheorem{proposition}[lemma]{Proposition}
\newtheorem{corollary}[lemma]{Corollary}
\newtheorem{claim}[lemma]{Claim}
\newtheorem{definition}[lemma]{Definition}

\newtheorem{manualtheoreminner}{Theorem}
\newenvironment{manualtheorem}[1]{%
  \renewcommand\themanualtheoreminner{#1}%
  \manualtheoreminner
}{\endmanualtheoreminner}
\usepackage[margin=1.1in]{geometry}
\usepackage{blindtext}

\usepackage[backend=biber, style=numeric, sorting=nyt]{biblatex}
\title{Automorphisms of very general blow up}
\author{\small{Supravat Sarkar}}
\date{}
\begin{document}

\begin{abstract}
   We show that if $X$ is a projective variety of dimension $\geq 2$ that is not a rational surface, then the blow up of $X$ at sufficiently many very general points has no nontrivial automorphism. Similar results were known before for $\mathbb{P}^2$ and $\mathbb{P}^3.$
\end{abstract}
\maketitle
\begin{center}
\end{center}
\begin{center}
\end{center}

\section{Introduction}
We work throughout over the field $\mathbb{C}$ of complex numbers. For a projective variety $X$, the symmetries of $X$ are encoded in its automorphism group $\text{Aut}(X)$. It is well-known that $\text{Aut}(X)$ itself has the structure of a group scheme, making it an interesting object of study. The study of automorphisms of varieties is a vast area of research. Automorphisms of several special classes of varieties have been studied by many authors. To name just a few, \cite{matsumura1963automorphisms},\cite{esser2024automorphisms} and \cite{esser2025hypersurfaces} study it for hypersurfaces in (weighted) projective spaces, \cite{lyu2024generic} for complete intersections in projective space, \cite{mcmullen2016automorphisms} for K3 surfaces, \cite{maruyama1971automorphism} for ruled surfaces, \cite{bansal2025automorphisms} for punctual Hilbert schemes, \cite{kollar2024automorphisms} for projective bundles, \cite{bansal2026isomorphisms} for multiprojective bundles, \cite{massarenti2013automorphisms} for moduli space of curves and \cite{kouvidakis1993automorphism} for moduli space of vector bundles on a curve.

There is a philosophy in Mathematics stating that whenever one has a large enough family of objects, the structure of automorphism group of a general member of the family is simple enough. \cite{ghadernezhad2015automorphism} can be thought of as a manifestation of this philosophy in logic/model theory, \cite[Theorem 1]{wright1974asymmetric} in graph theory, \cite{bhargava2025galois} in number theory, \cite{greene1982automorphism} in complex analysis, \cite{chapman2020asymptotic} in knot theory and combinatorics, \cite{bonatti2009c} in dynamical systems, \cite{ebin1970manifold} and \cite{mounoud2015metrics} in Riemannian geometry.

In algebraic geometry also, instances of this principle is abundant. The classical result stating that a general curve of genus $\geq 3$ has no nontrivial automorphism can be thought of as a manifestation of this philosophy. Another example is the generic triviality of automorphisms of hypersurfaces and complete intersections in projective space (with a few explicit exceptions), see \cite{matsumura1963automorphisms} and \cite{lyu2024generic}. From \cite[Theorem 0.2]{lyu2024generic} one obtains another result of the same flavour: if $X$ is a smooth projective variety of dimension $\geq 2$ and $D$ an ample divisor on $X$, then for all sufficiently large $d$, a general member of $|dD|$ has no nontrivial automorphism.

The goal of this paper is to prove an analogue of the above result, where instead of taking a general hypersurface in $X$, we do a birational transformation: we blow up a large collection of very general points in $X$. The main result of this paper is the following, which gives yet another manifestation of the above-mentioned philosophy.
\begin{manualtheorem}{A} \label{A}
Let $X$ be a projective variety of dimension $\geq 2$, and assume that
$X$ is not a rational surface. Then for all sufficiently large integers $r$ and
$p_1,\ldots,p_r\in X$ very general, the blow-up of $X$ at
$p_1,p_2,\ldots,p_r$ has no nontrivial automorphism.
\end{manualtheorem}

Though we believe Theorem \ref{A} should also hold when $X$ is a rational surface, our methods do not work in that case. For $X=\mathbb{P}^2$ this was proved in \cite{koitabashi1988automorphism} and \cite{hirschowitz1988symetries}. Recently in \cite{gachet2025pseudoautomorphism} a very similar problem was studied for $\mathbb{P}^3$, where instead of automorphism one considers pseudoautomorphism.

We give an application of this result in \S 6, constructing a smooth projective surface with trivial automorphism group and infinitely many elliptic bundle structures.
\section{Notation and preliminaries}
\begin{itemize}
\item For a positive integer $n$, we denote the set $\{1,2,\cdots,n\}$ by $[n]$.
\item For a normal projective variety $X$, the Picard rank of $X$ will be denoted by $\rho(X).$
\item For a smooth projective variety $X$, the canonical bundle of $X$ will be denoted by $\omega_X.$
\item For a birational morphism $f:X\to Y$ of normal projective varieties, we denote its exceptional locus by $\text{Ex}(f)\subset X$ and fundamental locus by $\text{Fundloc}(f)\subset Y.$
\item For a variety $X$ and a positive integer $r$, we say that $X$ and \textit{very general} points $p_1, p_2,\cdots ,p_r$ in $X$ satisfy a property if there is a countable union of proper subvarieties $\cup_iF_i\subsetneq X^r$ such that the property is satisfied whenever $(p_1,p_2,\cdots, p_r)\not\in \cup_iF_i$.
\item A morphism $f:X\to Y$ of projective varieties is called a \textit{smooth blow-up} if both $X$ and $Y$ are smooth and $f$ is the blow-up along a smooth subvariety of $Y$.
\item For points $p_1,p_2,\cdots,p_r$ in a variety $X$, we denote the blow-up of $X$ along $p_i$'s by $\operatorname{Bl}_{p_1,p_2,\cdots,p_r}X$.
\item The notation $r>>0$ will mean $r$ is a sufficiently large integer.
    \item We will use the notion of a $(-1)$-hypersurface, which can be thought of as a higher dimensional generalization of a $(-1)$-curve on a surface.
    \begin{definition}\label{-1 hypersurface}
Let $X$ be a normal projective variety of dimension $n$. A prime divisor
$E\subset X$ is called a $(-1)$-hypersurface if the following equivalent
conditions hold.

\begin{enumerate}
\item There is a Zariski neighborhood $V$ of $E$ in $X$, a smooth variety
$U$, and a morphism $V\to U$ which is the blow-up of $U$ at a smooth point
with exceptional divisor $E$.

\item Condition (1) holds with $V$ and $U$ analytic open.

\item $E\cong\mathbb{P}^{n-1}$ and lies in the smooth locus of $X$, with normal bundle
$N_{E/X}\cong\mathcal{O}_{\mathbb{P}^{n-1}}(-1)$.

\item There exists a morphism $X\to Y$ that is the blow-up at a smooth point of $Y$ with exceptional divisor $E$.
\end{enumerate}
\end{definition}

Here the implications \[(4)\Rightarrow
(1)\Rightarrow(2)\Rightarrow(3)
\] are clear, and $(3)\Rightarrow (4)$ can be proven using contraction theorem \cite[Theorem 3.7(3)]{KM} or \cite[Theorem 3.1]{andreatta2002special}. 

We need the following Lemma.
\begin{lemma}\label{contr hypersurface}
Let $X$ be a normal projective variety of dimension $\ge 3$. Then there
are only finitely many $(-1)$-hypersurfaces, and they are pairwise disjoint.
If $E_1,E_2,\ldots,E_r$
are $(-1)$-hypersurfaces in $X$, then there is a normal projective
variety $Y$ and a contraction $\pi:X\to Y$
which is the blow-up of $Y$ at $r$ smooth points with exceptional
divisor $\bigcup_i E_i.$
\end{lemma}

\begin{proof}
Suppose $E_1\neq E_2$
are $(-1)$-hypersurfaces in $X$, with $E_1\cap E_2\neq\varnothing.$
Let $f:X\to X_1$
be a smooth blow-up at a point with exceptional divisor $E_1$.
As $E_2\cong\mathbb{P}^{n-1}$
and $f|_{E_2}$
is nonconstant, $f|_{E_2}$
must be finite. So, $E_1\cap E_2$
is finite. But $E_1\cap E_2$
is an effective Cartier divisor in $E_2$, so has
positive dimension, as $\dim E_2=n-1\ge2$. This contradiction shows that $(-1)$-hypersurfaces in $X$ are disjoint.

The last statement
is clear by repeated application of (4) of Definition \ref{-1 hypersurface}. From this we see that the number
of $(-1)$-hypersurfaces is at most $\rho(X)-1$, hence finite.
\end{proof}
\item For a coherent sheaf $\mathcal{F}$ on a variety $X$, and a point $x\in X$, $\mathcal{F}(x)$ denotes the fibre of $\mathcal{F}$ a $x$, a finite dimensional $k(x)$-vector space. For a map $\mathcal{F}\xrightarrow{\phi}\mathcal{G}$ of coherent sheaves on $X$, we thus get an induced $k(x)$-linear map $\mathcal{F}(x)\xrightarrow{\phi(x)}\mathcal{G}(x)$.
\item We recall the notion of elementary transformation of ruled surfaces.
\begin{definition}\label{elementary transformation}
Let $\pi\colon X\longrightarrow C$
be a minimal ruled surface, and let $p_1,\ldots,p_r\in X$
be in distinct fibres of $\pi$. We define the elementary transformation of $X$ at $\underline{p}$ by $\operatorname{Elm}(X,\underline{p}):=(X',\underline{q}),$
constructed as follows.

Let $F_i=\pi^{-1}\pi(p_i)$, and $h\colon Y:=\operatorname{Bl}_{\underline{p}} X\longrightarrow X$ the blow-up map. The strict transform $\widetilde{F_i}$ of $F_i$ in $Y$ is a $(-1)$-curve. Let $h'\colon Y\longrightarrow X'$
be the contraction of all $\widetilde{F_i}$'s, and $q_i:=h'(\widetilde{F_i}).$ So, $\pi'\colon X'\longrightarrow C$
is again a minimal ruled surface, and $q_i$'s
are in distinct fibres of $\pi'$. Moreover, there is a natural birational map $h'\circ h^{-1}:X\dashrightarrow X'$. In summary, we have the following commutative diagram:
\[
\begin{tikzcd}[column sep=4em, row sep=3em]
& Y \arrow[dl,"h"'] \arrow[dr,"h'"] & \\
X \arrow[rr,dashed] \arrow[dr,"\pi"'] & & X' \arrow[dl,"\pi'"] \\
& C &
\end{tikzcd}
\]
\end{definition}
We record here a result from the theory of elementary transformations.
\begin{lemma}\label{maruyama}
Let $\mathcal{E}$ be a vector bundle of rank $2$ on a smooth projective curve $C$ and $X=\mathbb{P}_C(\mathcal{E})$.
Let $z_1, z_2,...,z_r$ be distinct points in $C$, and $\tau:\mathcal{E}\longrightarrow \oplus_ik(z_i)$
be a surjection, where $k(z_i)$ is regarded as a skyscraper sheaf at $z_i$. Note that $\ker\!\bigl(\tau(z_i):\mathcal{E}(z_i)\to k\bigr)$ is a
hyperplane in $\mathcal{E}(z_i),$
hence corresponds to a point $p_i\in\mathbb{P}(\mathcal{E}(z_i)).$ Let $\mathcal{F}:=\ker\tau,$
 a rank $2$ vector bundle and $\mathcal{F}\xrightarrow{\alpha}\mathcal{E}$ the sheaf inclusion. Let $q_i\in\mathbb{P}(\mathcal{F}(z_i))$
be the point corresponding to the hyperplane $\ker\!\bigl(\alpha(z_i):\mathcal{F}(z_i)\rightarrow\mathcal{E}(z_i)\bigr).$

Then there is a natural isomorphism 
\[
\operatorname{Elm}(\mathbb{P}(\mathcal{E}),\underline{p})
\cong
(\mathbb{P}(\mathcal{F}),\underline{q}),
\] over $C$,
and the natural birational map $\mathbb{P}(\mathcal{E})
\dashrightarrow
\mathbb{P}(\mathcal{F})$ as in Definition \ref{elementary transformation}
is same as $\alpha^*$.
\end{lemma}

\begin{proof}
See \cite{maruyama2006elementary}.
\end{proof}
\item For a $2$-dimensional $\mathbb{C}$-vector space $V$, we have the trace map
\[
\operatorname{tr}:\operatorname{End}(V)\longrightarrow \mathbb{C},
\]
and let $\mathfrak{sl}(V)$ denote its kernel. The short exact sequence
\[
0\longrightarrow \mathfrak{sl}(V)\longrightarrow
\operatorname{End}(V)
\xrightarrow{\operatorname{tr}}
\mathbb{C}
\longrightarrow 0
\]
has a natural splitting
\[
\mathbb{C}\longrightarrow \operatorname{End}(V),
\qquad
1\longmapsto \frac{1}{2}\operatorname{id}.
\]
Hence,
\[
\operatorname{End}(V)\cong
\mathbb{C}\oplus\mathfrak{sl}(V)
\]
naturally.

We also have a natural nondegenerate pairing
\[
\operatorname{End}(V)\times \operatorname{End}(V)
\longrightarrow
\mathbb{C},
\qquad
(\varphi,\psi)\longmapsto
\operatorname{tr}(\varphi\psi).
\]
Its restriction to $\mathfrak{sl}(V)$ is also nondegenerate. Consequently, we have natural isomorphisms
\[
\operatorname{End}(V)\cong \operatorname{End}(V)^*,
\qquad
\mathfrak{sl}(V)\cong \mathfrak{sl}(V)^*.
\]
Similarly, if $\mathcal{E}$ is a rank two vector bundle on a smooth projective curve $C$, then we have the split short exact sequence
\[
0\longrightarrow
\mathfrak{sl}(\mathcal{E})
\longrightarrow
\mathcal{E}nd(\mathcal{E})
\xrightarrow{\operatorname{tr}}
\mathcal{O}_C
\longrightarrow
0,
\]
and natural isomorphisms
\[
\mathcal{E}nd(\mathcal{E})
\cong
\mathcal{E}nd(\mathcal{E})^*,
\qquad
\mathfrak{sl}(\mathcal{E})
\cong
\mathfrak{sl}(\mathcal{E})^*.
\]
\item Let $\mathcal{E}$ and $L$ be a vector bundle and a line bundle on a smooth
projective curve $C$. If $\eta:\mathcal{E}\longrightarrow \mathcal{E}\otimes L$ is a map, then $\eta^2:\mathcal{E}\longrightarrow \mathcal{E}\otimes L^2$ will denote the composition
\[
\mathcal{E}
\xrightarrow{\eta}
\mathcal{E}\otimes L
\xrightarrow{\eta\otimes\operatorname{id}_L}
\mathcal{E}\otimes L^{2}.
\]

Similarly, if $
\eta:\mathcal{E}\otimes L\longrightarrow \mathcal{E}$
is a map, then $\eta^{2}:\mathcal{E}\otimes L^{2}\longrightarrow \mathcal{E}$ is
defined by the composition
\[
\mathcal{E}\otimes L^{2}
\xrightarrow{\eta\otimes\operatorname{id}_L}
\mathcal{E}\otimes L
\xrightarrow{\eta}
\mathcal{E}.
\]
\end{itemize}
\section{An auxiliary result}
The following result will play a crucial role in the proof of Theorem \ref{A}.
\begin{theorem}\label{permute}
Let $X$ be a normal projective variety, and let $Y$ be a proper closed subset of $X$. Then there is a nonnegative integer $r_0$ such that for all $r\ge r_0$ and $p_1,\ldots,p_r\in X$ very general, there is no nontrivial automorphism $\varphi$ of $X$ such that
\begin{equation}\label{contain}
\varphi(\{p_1,\ldots,p_r\})\subseteq Y\cup\{p_1,\ldots,p_r\}.
\end{equation} If $\dim X\geq 2$, we can in fact take $r_0=\dim\operatorname{Aut}^0(X)+1.$
\end{theorem}

\begin{proof}
Let $\operatorname{Aut}(X)$ be the automorphism group scheme of $X$. The scheme
$\operatorname{Aut}(X)\setminus\{\mathrm{id}\}$ has at most countably many connected
components; call them $G_1,G_2,\ldots$.

Let $n=\dim X$. Suppose $r\in\mathbb{N}$ is such that the conclusion of the Theorem does not hold. This more precisely means the following: given any countable union of proper subvarieties of $X^r$, there is a point $\underline{p}\in X^r$ outside that countable union and a nontrivial automorphism $\phi$ of $X$ such that \eqref{contain} is satisfied.

Let $Z_r=(X\setminus Y)^r\setminus\Delta$, where $\Delta$ is the diagonal. Let
\[
\Gamma_k=\{(\underline{p},\varphi)\in Z_r\times G_k\mid
\varphi(\{p_1,\ldots,p_r\})\subseteq
\{p_1,\ldots,p_r\}\cup Y\},
\]
a closed subset of $Z_r\times G_k$.

Let $f_k:\Gamma_k\to Z_r$ and $g_k:\Gamma_k\to G_k$ be the projections. If $\overline{f_k(\Gamma_k)}\neq Z_r$ for all $k$, then for
$\underline{p}\notin\bigcup_k\overline{f_k(\Gamma_k)}$, there is no nontrivial automorphism $\varphi$ of $X$ satisfying \eqref{contain}, a contradiction.

So there is $k$ with $\overline{f_k(\Gamma_k)}=Z_r$, hence
$\dim\Gamma_k\ge \dim Z_r=nr$. So, there is $\varphi\in G_k$ such that
\begin{equation}\label{dim}
\dim g_k^{-1}(\varphi)\ge nr-h^0(X,T_X),
\end{equation}
as $\dim G_k=\dim\operatorname{Aut}^0(X)=h^0(X,T_X)$.

For $S\subseteq [r]$, an injection $\eta:S\to [r]$, and $\phi\in G_k$, define
\[
W_{S,\eta,\phi}=
\left\{
\underline{p}\in Z_r\ \middle|\
\begin{array}{l}
\varphi(p_i)\in Y \text{ for all } i\notin S,\\
\varphi(p_i)=p_{\eta(i)} \text{ for all } i\in S
\end{array}
\right\}.
\] Note that
\[
g_k^{-1}(\varphi)\cong f_k(g_k^{-1}(\varphi))=\bigcup_{S,\eta}W_{S,\eta,\phi}.
\]
So there is $S,\eta$ such that
\[
\dim W_{S,\eta,\phi}\ge nr-h^0(X,T_X).
\]

Let $S\setminus\eta(S)=\{j_1,j_2,\ldots,j_t\}$, $[r]\setminus(S\cup\eta(S))=\{k_1,\ldots,k_m\}$. The map $\eta$ permutes the elements of the set
\[
S'=\{\,i\in S\mid \eta^l(i)\in S\text{ for all }l\ge1\,\},
\]
so decomposes $S'$ into cycles. Let $i_1,i_2,\ldots,i_s\in S'$ be representatives of distinct cycles. For $1\le u\le s$, let $a_u$ be the length of the cycle containing $i_u$, that is, the smallest natural number such that $\eta^{a_u}(i_u)=i_u$.

Without loss of generality, assume $a_1\le a_2\le\cdots\le a_s$.
Let $0\le s'\le s$ be the integer such that $a_u=1$ if and only if $u\le s'$. So we have 
\[
S'=\{i_1,\ldots,i_{s'}\}\cup
\{\eta^l(i_u)\mid s'<u\le s,\ 0\le l<a_u\}.
\]
Finally, for $1\le v\le t$, let $b_v\in\mathbb N$ be the unique integer such that $\eta^{b_v}(j_v)\notin S$, $\eta^l(j_v)\in S$ for all $0\le l<b_v$. Let $s''=s-s'$.

In summary, $[r]$ is partitioned in the disjoint union of the following sets:
\[
\{i_u\mid 1\le u\le s'\},\qquad
\{\eta^l(i_u)\mid s'<u\le s,\ 0\le l<a_u\},
\]
\[
\{\eta^l(j_v)\mid 1\le v\le t,\ 0\le l\le b_v\},\qquad
\{k_w\mid 1\le w\le m\}.
\]

So,
\[
r=s'+\sum_{s'<u\le s}a_u+\sum_{1\le v\le t}(b_v+1)+m.
\]

Hence,
\[
r-m-t=s'+\sum_{s'<u\le s}a_u+\sum_{1\le v\le t}b_v
\ge s'+2s''+t.
\]

Also,
\[
r\ge \sum_{s'<u\le s}a_u\ge 2s''.
\]

Also, the natural map
\[
W_{S,\eta,\phi}\longrightarrow
\prod_{u=1}^{s}X
\times
\prod_{v=1}^{t}X
\times
\prod_{w=1}^{m}X,
\]
given by
\[
\underline{p}\longmapsto
\bigl((p_{i_u})_u,\,(p_{j_v})_v,\,(p_{k_w})_w\bigr),
\]
is injective, and has image contained in

\[
W'=
\left\{
\bigl((x_u)_u,(y_v)_v,(z_w)_w\bigr)\
\middle|\
\begin{array}{l}
\varphi^{a_u}(x_u)=x_u,\\
\varphi^{\,b_v+1}(y_v)\in Y,\\
\varphi(z_w)\in Y
\end{array}
\text{ for all }u,v,w
\right\}.
\]
Note that
\[
W'\subseteq
\prod_{u=1}^{s'}\operatorname{Fix}(\varphi)
\times
\prod_{s'<u\le s}X
\times
\prod_{v=1}^{t}\varphi^{-(b_v+1)}(Y)
\times
\prod_{w=1}^{m}\varphi^{-1}(Y),
\]
where $\operatorname{Fix}(\varphi)=\{x\in X| \phi(x)=x\}.$

So,
\[
\begin{aligned}
nr-h^0(X,T_X)
&\le \dim W_{S,\eta,\phi}
\le \dim W' \\
&\le s'\dim\operatorname{Fix}(\varphi)
+s''n+(t+m)(n-1) \\
&\le (s'+t+m)(n-1)+s''n,
\end{aligned}
\]
the last inequality follows as $\operatorname{Fix}(\varphi)\neq X$ as $\phi$ is nontrivial.

So,
\[
\begin{aligned}
r
&\le h^0(X,T_X)+s''n+(n-1)(s'+t+m-r) \\
&\le h^0(X,T_X)+s''n-(n-1)(2s''+t) \\
&= h^0(X,T_X)-(n-2)s''-(n-1)t \\
&\le h^0(X,T_X)-(n-2)s''.
\end{aligned}
\]

For $n\ge2$, we thus get $r\le h^0(X,T_X)$. So we can take
$r_0=h^0(X,T_X)+1$.

Now assume $n=1$. So,
\[
r\le h^0(X,T_X)+s''
\le h^0(X,T_X)+\frac{r}{2}.
\]

Hence, $r\le 2h^0(X,T_X).$
So we can take $r_0=2h^0(X,T_X)+1.$
\end{proof}
\begin{corollary}\label{r'}
Let $X$ be a smooth projective surface and $m\geq 0$ an integer. Then there exists an integer $r_0'=r_0'(X,m)$ such that
for $r\ge r_0'$, and
$q_1,\ldots,q_m,p_1,\ldots,p_r\in X$ very general, $Y:=\operatorname{Bl}_{\underline{q}}(X)$
has no nontrivial automorphism fixing $\widetilde{p_i}$, where $\widetilde{p_i}$ is the preimage of $p_i$ in $Y$.
\end{corollary}
\begin{proof}
    Take $r_0'=h^0(X,T_X)+1=\dim \operatorname{Aut}^0(X)+1$. As $r_0'\geq \dim \operatorname{Aut}^0(Y)+1$, we are done by Theorem \ref{permute}.
\end{proof}
\section{Automorphism of very general blow up of ruled surface}
The goal of this section is to prove Proposition \ref{propo2}, which gives a weaker version of Theorem \ref{A} for ruled surfaces. We will use that to prove Theorem \ref{A} for all nonrational ruled surfaces in \S 5.
\begin{lemma}\label{Serre}
Fix a vector bundle $E$ on a smooth projective curve $C$. Then for all line bundles $L$ on $C$ of sufficiently large degree, we have
$h^1(C,E\otimes L)=0$.
\end{lemma}
\begin{proof}
Let $F=\omega_C\otimes E^*$. By Serre duality,
$$h^1(C,E\otimes L)=h^0(C,F\otimes L^{-1})
=\dim\operatorname{Hom}(L,F).$$
If $L$ has sufficiently large degree, then $h^0(C,L)>h^0(C,F)$, so there cannot be an injection of sheaves $L\to F$. Therefore,
$\operatorname{Hom}(L,F)=0$.
\end{proof}
\begin{lemma}\label{zls}
Let $C$ be a smooth projective curve, $\mathcal{E}$ a rank $2$ vector bundle on $C$. For a line bundle $L$ on $C$ and $0\neq s\in H^0(C,L^2)$ whose zero scheme $V(s)$ is nonempty and reduced, define
\[
Z_{L,s}=\{\eta\in\operatorname{Hom}(\mathcal{E},\mathcal{E}\otimes L)\mid \eta^2=\operatorname{id}\otimes s\},
\]
a Zariski closed subset of $\operatorname{Hom}(\mathcal{E},\mathcal{E}\otimes L)=H^0(C,\mathcal{E}nd(\mathcal{E})\otimes L)$.

Then the following hold:

\begin{enumerate}
\item For all $\eta\in Z_{L,s}$ and $z\in V(s)$, we have $\operatorname{rank}\eta(z)=1$.

\item $Z_{L,s}\subseteq H^0(C,\mathfrak{sl}(\mathcal{E})\otimes L)$.

\item There are positive integers $c_0=c_0(C,\mathcal{E})$ and $c_0'=c_0'(C,\mathcal{E})$ such that if $\deg L\ge c_0'$, then
\[
\dim Z_{L,s}\le \deg L+c_0.
\]
\end{enumerate}
\end{lemma}

\begin{proof}
The statements (1) and (2) can be checked locally near a point $z\in V(s)$. So assume $\mathcal{E}$ and $L$ are trivial, and $s\in\mathcal{O}_{z}$ a uniformizing parameter. Write $\eta$ as a matrix
\[
\begin{bmatrix}
f_1 & f_2\\
f_3 & f_4
\end{bmatrix},
\]
with $f_i\in\mathcal{O}_z$.

We have $\eta^2=s\cdot\operatorname{id}$. So,
\[
\begin{bmatrix}
f_1^2+f_2f_3 & f_2(f_1+f_4)\\
f_3(f_1+f_4) & f_4^2+f_2f_3
\end{bmatrix}
=
\begin{bmatrix}
s & 0\\
0 & s
\end{bmatrix}.
\]

If $f_1+f_4\neq 0$, then $f_2=f_3=0$ and $f_1^2=s$, which is impossible for $f_1\in\mathcal{O}_z$. Thus $f_1+f_4=0$, proving (2).

If $\eta(z)=0$, then $s\mid f_i$ for all $i$, so $s^2\mid f_1^2+f_2f_3=s$, a contradiction. So $\eta(z)\neq 0$. As $\eta(z)^2=0$, we get (1).

Now we prove (3). Let $L$ be any line bundle on $C$. Note that if $\eta\in H^0(C,\mathfrak{sl}(\mathcal{E})\otimes L)\subseteq\operatorname{Hom}(\mathcal{E},\mathcal{E}\otimes L)$, then
\[
\eta^2\in\operatorname{Hom}(\mathcal{E},\mathcal{E}\otimes L^2)
=H^0(C,\mathcal{E}nd(\mathcal{E})\otimes L^2)
\]
lies in the subspace $H^0(C,L^2)$. This can be proved by a local computation similar to the proof of (1) and (2). Here we regard $L^2$ as a subbundle of $\mathcal{E}nd(\mathcal{E})\otimes L^2$ via the nowhere vanishing section of $\mathcal{E}nd(\mathcal{E})$ given by the identity endomorphism.

So, we get a morphism
\[
\psi:H^0(C,\mathfrak{sl}(\mathcal{E})\otimes L)\longrightarrow H^0(C,L^2),
\qquad
\eta\longmapsto \eta^2,
\]
and $Z_{L,s}=\psi^{-1}(s)$ for $0\neq s\in H^0(C,L^2)$ whose zero scheme is nonempty reduced.

The natural isomorphism $\mathfrak{sl}(\mathcal{E})\cong \mathfrak{sl}(\mathcal{E})^*$ induces a natural isomorphism
\[
\mathfrak{sl}(\mathcal{E})\otimes L\cong \mathcal{H}om(\mathfrak{sl}(\mathcal{E}),L).
\]
So, any $\eta\in H^0(C,\mathfrak{sl}(\mathcal{E})\otimes L)$ induces a map
\[
\widehat{\eta}:\mathfrak{sl}(\mathcal{E})\longrightarrow L.
\]Let $M_\eta:=\ker\widehat{\eta}$.

\begin{claim}\label{etahat}
Let $\eta\in Z_{L,s}$. Then the following hold:

\begin{enumerate}
\item[(a)] $\widehat{\eta}$ is surjective.

\item[(b)] $(d\psi)_\eta$ is same as the map on $H^0$ induced by $\widehat{\eta}\otimes\operatorname{id}:\mathfrak{sl}(\mathcal{E})\otimes L\longrightarrow L^2.$

\item[(c)]
$\operatorname{rank}(d\psi)_\eta
=
h^0(C,L^2)-h^1(C,M_\eta\otimes L), $ if $h^1(C,\mathfrak{sl}(\mathcal{E})\otimes L)=0$,
\end{enumerate}
\end{claim}
\begin{proof}
$(a)$ This can be checked locally near any point $z\in C$.
So assume $\mathcal{E}$ and $L$ are trivial.
 $\eta$ is given by a matrix
\[
\begin{pmatrix}
f_1 & f_2\\
f_3 & -f_1
\end{pmatrix},
\]
where $f_i\in\mathcal{O}_{z}$,
and $
\mathfrak{sl}(\mathcal{E})
\xlongrightarrow{\widehat{\eta}}
L$
is given by
\[
\mathcal{O}_{z}^{\,3}
\longrightarrow
\mathcal{O}_{z},
\]
\[
(g_1,g_2,g_3)
\longmapsto
2g_1f_1+g_2f_3+g_3f_2.
\]

Since $\eta^2=\operatorname{id}\otimes s$, We have $\operatorname{ord}_z(f_1^2+f_2f_3)\le 1.$
So some $f_i$ does not vanish at $z$. Hence $\widehat{\eta}$
is surjective at $z$.

$(b)$ Note that
\[
(d\psi)_\eta(\eta')
=
\eta\eta'
+
\eta'\eta,
\]
for $\eta'\in H^0(C,\mathfrak{sl}(\mathcal{E})\otimes L).$

Here $\eta\eta',\ \eta'\eta$ are elements of 
$H^0\!\left(C,\mathcal{E}nd(\mathcal{E})\otimes L^2\right),$
and their sum lies in $H^0(C,L^2)
\subset
H^0\!\left(C,\mathcal{E}nd(\mathcal{E})\otimes L^2\right).$

So we need to show that
\[
(\widehat{\eta}\otimes\operatorname{id})(\eta')
=
\eta\eta'
+
\eta'\eta.
\]

This can be checked by a straightforward local computation.

\medskip

\noindent
$(c)$ Follows immediately from $(b)$ and the short exact sequence
\[
0
\longrightarrow
M_\eta\otimes L
\longrightarrow
\mathfrak{sl}(\mathcal{E})\otimes L
\longrightarrow
L^2
\longrightarrow
0.
\]
\end{proof}
Now we finish the proof of the Lemma. Let $\mathcal{F}=\mathfrak{sl}(\mathcal{E})$, and $q:\mathbb{P}_C(\mathcal{F})\to C$ be the projection. There is a line bundle $N$ on $C$ such that $q_*\Omega_q(2)\otimes N$ is globally generated, hence $q^*q_*\Omega_q(2)\otimes q^*N$ is globally generated. Note that the restriction of $\Omega_q(2)$ to each fibre of $q$ is globally generated, as $\Omega_{\mathbb{P}^2}(2)$ is globally generated. Therefore, by cohomology and base change, the natural map $q^*q_*\Omega_q(2)\longrightarrow \Omega_q(2)$
is surjective. Consequently, $\Omega_q(2)\otimes q^*N$ is globally generated.

By Lemma \ref{Serre}, we can choose $c_0'\in\mathbb{N}$ such that the following hold:

\begin{enumerate}
\item For any line bundle $L$ on $C$ of degree $\ge c_0'$, we have
$h^1(C,\mathcal{F}\otimes L)=h^1(C,L^2)=0$.

\item For any line bundle $N''$ of degree $\ge c_0'+\deg N+\deg\mathcal{F},$
we have $h^1(C,N'')=0$.
\end{enumerate}

Suppose $\eta\in Z_{L,s}$, where $\deg L\ge c_0'$. The surjection
$\widehat{\eta}$ gives a section $\sigma$ of $q$ with $\sigma^*\mathcal{O}_{\mathbb{P}(\mathcal{F})}(1)\cong L$. The Euler exact sequence
gives an exact sequence on $\mathbb{P}(\mathcal{F})$:
\[
0\longrightarrow \Omega_q(2)\longrightarrow q^*(\mathcal{F})(1)\longrightarrow \mathcal{O}_{\mathbb{P}(\mathcal{F})}(2)\longrightarrow 0.
\]

Pulling back by $\sigma$, we get the following short exact sequence on $C$:
\[
0\longrightarrow \sigma^*\Omega_q(2)\longrightarrow
\mathcal{F}\otimes L
\xrightarrow{\ \widehat{\eta}\otimes\operatorname{id}\ }
L^2
\longrightarrow 0.
\]

This shows $M_\eta\otimes L\cong\sigma^*\Omega_q(2)$. Hence $M_\eta\otimes L\otimes N$ is globally generated. Therefore a general section of $M_\eta\otimes L\otimes N$ is nowhere vanishing. So we have a short exact sequence of vector bundles
\[
0\longrightarrow \mathcal{O}_C
\longrightarrow
M_\eta\otimes L\otimes N
\longrightarrow
N'
\longrightarrow 0,
\]
hence a short exact sequence of vector bundles
\begin{equation}\label{N''}
0\longrightarrow N^{-1}
\longrightarrow
M_\eta\otimes L
\longrightarrow
N''
\longrightarrow 0.
\end{equation}

Note that $N''$ is a line bundle of degree
\begin{align*}
\deg(M_\eta\otimes L)+\deg N
&=
\deg(\mathcal{F}\otimes L)-2\deg L+\deg N\\
&=
\deg\mathcal{F}+\deg L+\deg N\\
&\ge
c_0'+\deg N+\deg\mathcal{F},
\end{align*}
so $h^1(C,N'')=0$. Hence, \eqref{N''} gives
$h^1(C,M_\eta\otimes L)\le h^1(C,N^{-1})$. By Claim \ref{etahat}(c),
\[
\operatorname{rank}(d\psi)_\eta
\ge
h^0(C,L^2)-h^1(C,N^{-1}).
\]

Finally,
\begin{align*}
\dim_{\eta} Z_{L,s}
&\le
\dim T_\eta(Z_{L,s})
\le \dim\ker\!\left(d\psi\right)_\eta\\
&= h^0(C,\mathcal{F}\otimes L)-\operatorname{rank}(d\psi)_\eta\\
&\le
h^0(C,\mathcal{F}\otimes L)
-
h^0(C,L^2)+h^1(C,N^{-1})\\
&=
\chi(C,\mathcal{F}\otimes L)
-
\chi(C,L^2)
+
h^1(C,N^{-1})\\
&=
\deg\mathcal{F}
+3\deg L
+3(1-g)
-
2\deg L
-(1-g)
+
h^1(C,N^{-1})\\
&=
\deg L
+\deg\mathcal{F}
+h^1(C,N^{-1})
+2(1-g).
\end{align*}

Hence one may take
\[
c_0
=
\deg\mathcal{F}
+h^1(C,N^{-1})
+2(1-g).
\]
\end{proof}

\begin{proposition}\label{Propo1}
Fix a minimal ruled surface $\pi:X\to C$ over a smooth projective curve $C$, and an integer $d\ge0$. Then there is a positive integer $r_0=r_0(X,d)$ such that the following holds:

Fix $r\ge r_0$, distinct points $z_1,\ldots,z_{d+r}\in C$, and
$p_i\in\pi^{-1}(z_i)$ for $1\le i\le d$. Let
$(p_{d+1},\ldots,p_{d+r})\in\prod_{d<i\le r}\pi^{-1}(z_i)$
be general, and
$Y=\operatorname{Bl}_{p_1,\ldots,p_{d+r}}X$, with exceptional divisors $E_i$ over $p_i$ for $1\leq i\leq d+r.$
Then any involution of $Y$ over $C$ preserves some $E_i$.
\end{proposition}
\begin{proof}
   Let $X=\mathbb{P}_C(\mathcal{E})$, where $\mathcal{E}$ is a rank $2$ vector bundle on $C$. Let $c_0(C,\mathcal{E})$ and $c_0'(C,\mathcal{E})$ be as in Lemma \ref{zls}, and let
\[
r_0=2\bigl(c_0(C,\mathcal{E})+c_0'(C,\mathcal{E})+d\bigr).
\]
We show it works. 

Let $r\ge r_0$ and $z_1,\ldots,z_{d+r}\in C$ be distinct points, $p_i\in\pi^{-1}(z_i)$ for $1\le i\le d+r$, $Y=\operatorname{Bl}_{\underline{p}} X$, and $h:Y\to X$ the blow-up map, with $E_i$ the exceptional divisors over $p_i$.
Suppose $\varphi$ is an involution of $Y$ over $C$ with $\varphi(E_i)=:E_i'\ne E_i$ for all $i$.

Let $Y\xrightarrow{h'}X'$ be the contraction of all the $E_i'$'s, and let $p_i'=h'(E_i)$. So, $\operatorname{Elm}(X,\underline{p})\cong (X',\underline{p'})$,
and $\varphi$ induces an isomorphism $(X,\underline{p})\xrightarrow{\;\overline{\varphi}\;}(X',\underline{p'})$
over $C$.

Let $\mathcal{E}\xrightarrow{\tau}\bigoplus_i k(z_i)$
be a surjection such that $p_i\in\mathbb{P}(\mathcal{E}(z_i))$ corresponds to the hyperplane $\ker\tau(z_i)$, and let $\mathcal{F}$, $\alpha$, and $\underline{q}$ be as in the Lemma \ref{maruyama}. By Lemma \ref{maruyama}, we can identify $(X',\underline{p'})=(\mathbb{P}(\mathcal{F}),\underline{q})$
over $C$. So, we have an isomorphism $\overline{\varphi}:\mathbb{P}(\mathcal{E})\longrightarrow \mathbb{P}(\mathcal{F})$
over $C$ with $\overline{\varphi}(p_i)=q_i$. By \cite[Exercise II.7.10(d)]{hartshorne2013algebraic}, $\overline{\varphi}$ is induced by an isomorphism of vector bundles $\mathcal{E}\otimes L^{-1}\xrightarrow{\psi}\mathcal{F}$,
where $L$ is a line bundle on $C$. As $\overline{\varphi}(p_i)=q_i$, we have
\[
\psi(z_i)\bigl(\ker\tau(z_i)\otimes L^{-1}(z_i)\bigr)=\ker\alpha(z_i),
\]
hence
\[
\ker(\alpha\psi)(z_i)
=
\ker\tau(z_i)\otimes L^{-1}(z_i)
=
\operatorname{im}\alpha\psi(z_i)\otimes L^{-1}(z_i),
\]
where the last equality follows from the short exact sequence
\begin{equation}\label{alpha psi}
0\longrightarrow
\mathcal{E}\otimes L^{-1}
\xrightarrow{\ \alpha\psi\ }
\mathcal{E}
\xrightarrow{\ \tau\ }
\bigoplus_i k(z_i)
\longrightarrow0.
\end{equation}

So, if $(\alpha\psi)^2:\mathcal{E}\otimes L^{-2}\to\mathcal{E}$ is the induced map, then $(\alpha\psi)^2(z_i)=0$
for all $i$.

Note that by \eqref{alpha psi}, $L^{2}\cong\mathcal{O}_C\!\left(\sum_i z_i\right).$
Let $s$ be a section of $L^2$ with zero divisor $\sum_i z_i$. Let
$\eta:\mathcal{E}\to\mathcal{E}\otimes L$ and
$\eta^2:\mathcal{E}\to\mathcal{E}\otimes L^2$
be the maps induced by $\alpha\psi$ and $(\alpha\psi)^2$, respectively.
So, $\eta^2(z_i)=0$ for all $i$.

\begin{claim}

There is $\lambda\in\mathbb{C}^{\times}$ such that
$\eta^2=\lambda\,\mathrm{id}\otimes s$.
    
\end{claim}

\begin{proof}
We have a commutative diagram by Lemma \ref{maruyama}:
\[
\begin{tikzcd}[column sep=large,row sep=large]
Y \arrow[r,"\varphi"] \arrow[d,"h"'] &
Y \arrow[d,"h'"] \arrow[dr,"h"] & \\
X &
X' \arrow[l,"\psi^*"'] &
X \arrow[l,densely dotted,"\alpha^*"']
\end{tikzcd}
\]

Hence
\[
h\varphi^{-1}h^{-1}
=
\psi^*\alpha^*
=
(\alpha\gamma)^*
\] in $\operatorname{Bir}(X)$. As $\varphi^2=\mathrm{id}$, we get $\bigl((\alpha\psi)^*\bigr)^2=\mathrm{id}$ in $\operatorname{Bir}(X)$,
so $\bigl((\alpha\psi)^2\bigr)^*=\mathrm{id}.$ Hence $(\eta^2)^*=\mathrm{id}$. Therefore,
\[
\eta^2=\lambda\,\mathrm{id}\otimes s
\]
for some $\lambda\in K(C)$.

Since the points $z_i$'s are distinct, $\operatorname{id}\otimes s \in H^0(C,\mathcal{E}nd(\mathcal{E})\otimes L^2)$ vanishes exactly at the points $z_i$ and has simple zeroes there. As $\eta^2(z_i)=0$ for all $i$, we get $\lambda\in H^0(C,\mathcal{O}_C)=\mathbb{C}$. Finally, $\lambda\neq 0$, as $\eta^2(z)\neq 0$ for $z\notin\{z_1,\ldots,z_r\}$, $\eta(z)$ being an isomorphism.
\end{proof}

Let $\sqrt{\lambda}$ be a square root of $\lambda$, so $\frac{1}{\sqrt{\lambda}}\eta\in Z_{L,s}$. By Lemma \ref{zls}, $\operatorname{rank}\eta(z_i)=1$ for all $i$, hence $\ker\eta(z_i)=\operatorname{im}(\alpha\psi)(z_i)=\ker\tau(z_i)$, that corresponds to the $p_i\in\mathbb{P}(\mathcal{E}(z_i))$. Consider the map
\[
\kappa_{L,s}: Z_{L,s}\longrightarrow \prod_{i=1}^{d+r}\mathbb{P}(\mathcal{E}(z_i)),
\]
given by
\[
\eta'\longmapsto \bigl(\ker\eta'(z_i)\bigr)_i.
\]

Here $\ker\eta'(z_i)$ is a hyperplane in $\mathcal{E}(z_i)$ Lemma \ref{zls}, so corresponds to a point in $\mathbb{P}(\mathcal{E}(z_i))$. Note that
\[
\kappa_{L,s}\!\left(\frac{1}{\sqrt{\lambda}}\eta\right)
=(p_1,\ldots,p_{d+r}).
\]
So,
\begin{equation}\label{im k}
(p_1,\ldots,p_{d+r})\in \overline{\operatorname{im}(\kappa_{L,s})}.
\end{equation}

Note that upto isomorphism there are only finitely many pairs $(L,[s])$
of $[s]\in \mathbb{P}\!\left(H^0(C,L^2)^*\right)$
with $[V(s)]=\sum_i z_i,$
and \(\operatorname{im}(\kappa_{L,s})\) depends only on \((L,[s])\). If for
\[
(p_{d+1},\ldots,p_{d+r})
\in
\prod_{d<i\le r}\pi^{-1}(z_i)
\]
general, \eqref{im k} is satisfied for some \((L,s)\), then we must have
\[
r-d
\le
\dim\overline{\operatorname{im}(\kappa_{L,s})}
\]
for some \((L,s)\).

As
\[
\deg L=\frac r2>c_0'(C,\mathcal{E}),
\]
we have by Lemma \ref{zls},
\[
\dim\overline{\operatorname{im}(\kappa_{L,s})}
\le
\dim Z_{L,s}
\le
\frac r2+c_0(C,\mathcal{E}).
\]

So, $r-d
\le
\frac r2+c_0(C,\mathcal{E}),$
hence $r
\le
2d+2c_0(C,\mathcal{E}),$ a
contradiction to the choice of \(r\).
\end{proof}
\begin{proposition}\label{propo2}
Fix a minimal ruled surface $\pi:X_0\to C$ over a smooth projective curve $C$, an integer $s\ge0$, and points $q_1,\ldots,q_s\in X_0$ in distinct fibres of $\pi$. Let $X=\operatorname{Bl}_{\underline{q}}X_0$. Then for all $r\gg0$, the blow-up of $X$ at $r$ very general points has no nontrivial automorphism over $C$.
\end{proposition}

\begin{proof}
Let $X\xrightarrow{\eta_k}X_k$ be the contractions of $X$ over $X_0$, for
$0\le k\le n$. For $0\leq i\leq s$, let $r_0(X_0,s)$ be as in Proposition \ref{Propo1}, and $A$ be the maximum of them. Let $r_0'(X_k,j)$ be as in Corollary \ref{r'} for $0\leq k\leq n$ and $0\leq j\leq A$. Choose
\[
r>A+\sum_{j=0}^{A}\sum_{k=0}^{n}r_0'(X_k,j).
\]
We show this works.

Let $w_i=\pi(q_i)$, and let $p_1,\ldots,p_r\in X$ be very general. Let $Y=\operatorname{Bl}_{\underline{p}}X$, and suppose $Y$ has a nontrivial automorphism $\varphi$. We want to get a contradiction.

Let $E_i'\subset X$ be the exceptional divisor over $q_i$ for $1\le i\le s$, and $E_j\subset Y$ be the exceptional divisor over $p_j$ for $1\le j\le r$. Note that $Y\to X$ is an isomorphism over $E_i'$. By abuse of notation we will denote the inverse image of $E_i'$ in $Y$ again by $E_i'$. Since $p_j$'s are very general, all $E_i'$ and $E_j$ lie in distinct fibres of $Y$ over $C$.

Without loss of generality, assume $\varphi(E_i')=E_i'$ if and only if $1\le i\le s'$, and $\varphi(E_j)=E_j$ if and only if $1\le j\le r'$, where $0\le s'\le s$ and $0\le r'\le r$ are integers. Since the fibres of $Y$ over $C$ have at most $2$ irreducible components, $\varphi^2$ preserves the irreducible components of each fibre, hence preserves each $E_j$. So, $\varphi^2$ descends to an automorphism of $X$ fixing each $p_j$, so $\varphi^2=\mathrm{id}$ as $r>r_0'(X,0)$. Hence $\varphi$ is an involution.

Let $0\le k\le n$ be such that $\eta:=\eta_k:X\to X_k$ is the contraction of $E_1',\ldots,E_{s'}'$. Let $Y\to Y'$ be the contraction of $E_1',\ldots,E_{s'}',E_1,\ldots,E_{r'}$, with $E_j$ contracted to $\widetilde{\eta(p_j)}$. Thus we have an induced map $Y'\to X_k$ which is the blow-up along the points $\eta(p_j)$ for $r'<j\le r$. As the $p_j$'s are very general in $X$, the points $\eta(p_j)$ are very general in $X_k$. Also, $\varphi$ induces a nontrivial involution $\varphi'$ of $Y'$ with $\varphi'(\widetilde{\eta(p_j)})=\widetilde{\eta(p_j)}$ for $1\le j\le p'$. So, by Lemma \ref{permute}, we have
\begin{equation}\label{rr01}
r' < r_0'(X_k,r-r'). 
\end{equation}

Also, by the construction of $Y'$, we see that the involution $\varphi'$ does not preserve any $(-1)$-curve lying in a fibre over $C$. As $Y'$ is the blow-up of $X_0$ at the points $q_i$ ($s'<i\le s$) and $\eta_0(p_j)$ ($r'<j\le r$), and the points $\eta_0(p_j)$ are very general, by Proposition \ref{Propo1} we get
\begin{equation}\label{rr02}
r-r' < r_0(X_0,s-s') \le A.
\end{equation}

Now \eqref{rr01} and \eqref{rr02} give
\[
r <A + r_0'(X_k,r-r')
   \le A+\sum_{j=0}^{A}\sum_{k=0}^{n} r_0'(X_k,j)<r,
\]
a contradiction. 
\end{proof}
\section{Proof of Theorem \ref{A}}
Let $\overline{X}$ be the normalization of $X$, and let
$\overline{p}_1,\ldots,\overline{p}_r$ be the preimages of
$p_1,\ldots,p_r$ on $\overline{X}$. Note that the normalization of
$\operatorname{Bl}_{\underline{p}}X$ is $\operatorname{Bl}_{\underline{\overline{p}}}\overline{X}$,
hence any automorphism of $\operatorname{Bl}_{\underline{p}}X$ induces an automorphism
of $\operatorname{Bl}_{\overline{\underline{p}}}\overline{X}$. So, replacing $X$ by
$\overline{X}$, We can assume $X$ is normal.

If $X$ is a surface, by the same argument we can replace $X$ by its minimal
resolution to assume $X$ is smooth.

Let $\widetilde{X}=\operatorname{Bl}_{\underline{p}}X$,
$\pi:\widetilde{X}\to X$ the blow-up map, and
$\varphi\in\operatorname{Aut}(\widetilde{X})$.

Now we consider three cases separately.

\noindent\textbf{Case 1:} $\dim X\ge 3$.

Let $X\xlongrightarrow{g}Y$ be the contraction of all $(-1)$-hypersurfaces of
$X$, as in Lemma \ref{contr hypersurface}. Since by Lemma \ref{contr hypersurface} any $(-1)$-hypersurface of
$\widetilde{X}$ is either contained in $\operatorname{Ex}(\pi)$ or disjoint
from $\operatorname{Ex}(\pi)$, we see that
$g\circ\pi:\widetilde{X}\to Y$ is the contraction of all
$(-1)$-hypersurfaces of $\widetilde{X}$, as in Lemma \ref{contr hypersurface}.

So, $\varphi$ descends to an automorphism $\psi$ of $Y$. The finite set $\operatorname{Fundloc}(g\circ\pi)$ is $\psi$-stable, hence
$\psi(\{g(p_1),\ldots,g(p_r)\})\subseteq
\{g(p_1),\ldots,g(p_r)\}\cup\operatorname{Fundloc}(g)$.

As $g(p_1),\ldots,g(p_r)\in Y$ are very general, by Theorem \ref{permute} we get
$\psi=\operatorname{id}$ whenever $r\gg0$. So, $\varphi=\operatorname{id}$.

\medskip

\noindent\textbf{Case 2:} $X$ is a non-uniruled surface.

 Let $X\xrightarrow{g}Y$ be the minimal model of $X$.
The birational map $\psi=(g\pi)\circ\varphi\circ(g\pi)^{-1}:Y\dashrightarrow Y$
is an isomorphism by \cite[Theorem 3.52(2)]{KM}.

The rest is exactly similar as in Case 1.

\noindent\textbf{Case 3:} $X$ is a uniruled surface.

As $X$ is not rational, there is a ruled surface structure $\pi:X\to C$ over a smooth nonrational curve $C$. Let $S\subset C$ be the finite set of points over which the fibre of $\pi$ is not irreducible. Let $p_1,\ldots,p_r\in X$ be very general, $Y=\operatorname{Bl}_{\underline{p}} X$, and $\pi':Y\to X\xrightarrow{\pi}C$ the natural map. Any automorphism $\varphi$ of $Y$ descends to an automorphism $\overline{\varphi}$ of $C$, as the fibres of $\pi'$ are connected and unions of rational curves, so has no nonconstant morphism to $C$. Further, $\overline{\varphi}$ preserves the set
 $S\cup\{\pi(p_1),\ldots,\pi(p_r)\}\subset C$, which is the set of points over which
the fibres of $\pi'$ are not irreducible.

If $p_1,\ldots,p_r$ are very general, so are $\pi(p_1),\ldots,\pi(p_r)$, so by Lemma \ref{permute}, $\overline{\varphi}=\mathrm{id}$ if $r\gg0$. Thus $\varphi$ is an automorphism over $C$.

Let $S'\subseteq S$ be the set of points over which the fibre of $\pi$ has $\geq 3$ irreducible components. An easy application of \cite[Proposition 1.18]{hassett2009rational} shows that the $(-1)$-curves in $X$ lying in $\pi^{-1}(S')$ are disjoint, let $X\xrightarrow{\nu}X'$ be the contraction of all of them. Similarly, the $(-1)$-curves in $Y$ lying in $(\pi')^{-1}(S')$ are also disjoint, let $Y\to Y'$ be the contraction of all of them. Also, $\varphi$ permutes these $(-1)$-curves, so $\varphi$ induces a nontrivial automorphism of $Y'$ over $C$. Note that $\eta(p_i)$'s are also very general, and $Y'$ is the blow-up of $X'$ at the points $\eta(p_i)$'s. So, replacing $X$ by $X'$, we may assume $S'=S$, that is, each fibre of $\pi$ has at most $2$ irreducible components.

In this case, there is a minimal ruled surface $\pi_0:X_0\to C$ and points $q_1,\ldots,q_s\in X_0$ lying in distinct fibres of $\pi_0$, such that $X=\operatorname{Bl}_{\underline{q}} X_0$. Now we are done by Proposition \ref{propo2}.

This completes the proof.
\begin{remark}
Theorem \ref{A} does not hold if we assume $p_1,p_2,\ldots,p_r$ are general, rather than very general, as the following example shows. This is in contrast to the generic triviality of automorphism group of hypersurface mentioned in the introduction, where it suffices to choose the hypersurface to be general.

Let \(X\) be an abelian variety of dimension \(\ge 2\) such that there is $\varphi\in\operatorname{Aut}(X)$
of infinite order. Let \(r\) be any positive integer. Since the torsion points of the abelian variety \(X^r\) are Zariski dense, given any  Zariski closed subset $Z\subsetneq X^r,$
we can find torsion points $p_1,\ldots,p_r\in X$
such that $(p_1,\ldots,p_r)\notin Z.$
There is \(N\in\mathbb{N}\) with $p_1,\ldots,p_r\in X[N],$
the finite group of \(N\)-torsion points of $X$.
Since \(\varphi\) preserves \(X[N]\), some iterate $\psi$
of \(\varphi\) is a nontrivial automorphism of \(X\) fixing \(X[N]\) pointwise, hence fixing each \(p_i\). Therefore, \(\psi\) induces a nontrivial automorphism of $\operatorname{Bl}_{p_1,\ldots,p_r}X.$
\end{remark}
\begin{remark}
Theorem \ref{A} does not hold if we drop the assumption that \(X\) is projective. As an example, let $X=\mathbb{A}^3$ $p_i=({q_i},\alpha_i),$
where \({q_i}\in \mathbb{A}^2\), \(\alpha_i\in \mathbb{A}^1\), and assume the \(q_i\)'s are distinct. There is $P(x,y)\in\mathbb{C}[x,y]$
with $P(q_i)=\alpha_i$
for all \(i\). Now define $\varphi\in\operatorname{Aut}(\mathbb{A}^3)$
by
\[
\varphi(q,\alpha)=\bigl(q,\;2P(q)-\alpha\bigr),
\qquad
(q,\alpha)\in \mathbb{A}^2\times \mathbb{A}^1=\mathbb{A}^3.
\]
Note that $\varphi(p_i)=p_i$
for all \(i\). Hence \(\varphi\) induces a nontrivial automorphism of $\operatorname{Bl}_{p_1,\ldots,p_r}(\mathbb{A}^3).$
\end{remark}
\section{An application}
Let $X$ be a smooth projective variety. Using \cite[Theorem 2.2]{wisniewski1991contractions}, it is easy to see that $X$ can have only finitely many projective bundle structures. When $X$ is a surface, it is natural to ask whether $X$ has only finitely many elliptic bundle structures. When $K_X\not\equiv 0$, using canonical bundle formula, it is easy to see that $X$ can have at most one minimal elliptic bundle structures. If $K_X\equiv 0$, $X$ can in fact have infinitely many minimal elliptic bundle structures as some abelian and K3 surfaces show, but by \cite[Theorem 4.1]{Totaro2010ConeConjecture} $X$ has only finitely many minimal elliptic bundle structures upto automorphisms of $X$. Thus in any case $X$ has only finitely many minimal elliptic bundle structures upto automorphisms of $X$.

Now let us consider possibly non-minimal elliptic bundle structures on $X$. In the following example we show that $X$ can have infinitely many elliptic bundle structures which are distinct modulo automorphims of $X$.

Let $E$ be an elliptic curve, $A=E^2$, an abelian surface, and let $X$ be the blow-up of $A$ at sufficiently many very general points. By Theorem \ref{A}, $X$ has no nontrivial automorphism. On the other hand, the action of $GL(2,\mathbb{Z})$ on $A$ shows that $A$ has infinitely many elliptic bundle structures, hence composing them with the blow-up map $X\to A$, we see that $X$ also has infinitely many elliptic bundle structures.

Though Theorem \ref{A} is not known for rational surfaces, by some argument similar as above but slightly more delicate, one can also construct a smooth projective rational surface $X$ with infinitely many elliptic bundle structures which are distinct modulo automorphims of $X$, by looking at some blow up of an unnodal Coble surface as in \cite{cantat2012rational}.

\section{Acknowledgement}
I thank János Kollár, John Lesieutre and Daniel Litt for insightful discussions.
\printbibliography
\vspace{40pt}
\begin{flushleft}
{\scshape Department of Mathematics, Fine Hall, Princeton University, Princeton, NJ 700108, USA}.

{\fontfamily{cmtt}\selectfont
\textit{Email address: ss6663@princeton.edu} }
\end{flushleft}
\end{document}